\documentclass[a4paper,10pt,leqno, english]{amsart}
\title[A Probabilistic approach]{A probabilistic approach  to  exhaustion in the infinite-genus case}

\author{No\'{e} B\'{a}rcenas }
                \email{barcenas@matmor.unam.mx}
         \urladdr{http://www.matmor.unam.mx /~ barcenas}
         
\author{Ricardo E.  Ch\'avez-C\'aliz}
       \email{ricardo.ch.cz@gmail.com}

\author{Jes\'us Hern\'andez Hern\'andez}
       \email{jhdez@matmor.unam.mx}
       \urladdr{https://sites.google.com/site/jhdezhdez/}
 \address{Centro de Ciencias Matem\'aticas. UNAM \\Ap.Postal 61-3 Xangari. Morelia, Michoac\'an MEXICO 58089}

         \date{\today}

\usepackage{hyperref}
\usepackage{pdfsync}
\usepackage{calc}
\usepackage{enumerate,amssymb}
\usepackage[arrow,curve,matrix,tips,2cell, arc, all]{xy}
  \SelectTips{eu}{10} \UseTips
  \UseAllTwocells

\usepackage{color}
\usepackage{graphicx}

\DeclareMathAlphabet\EuR{U}{eur}{m}{n}
\SetMathAlphabet\EuR{bold}{U}{eur}{b}{n}

\makeindex             

\theoremstyle{plain}
\newtheorem{theorem}{Theorem}[section]
\newtheorem{lemma}[theorem]{Lemma}
\newtheorem{proposition}[theorem]{Proposition}

\theoremstyle{definition}
\newtheorem{definition}[theorem]{Definition}

\newtheorem{condition}[theorem]{Condition}

{\catcode`@=11\global\let\c@equation=\c@theorem}

\newcommand{\comsquare}[8]                   
{\begin{CD}
#1 @>#2>> #3\\
@V{#4}VV @V{#5}VV\\
#6 @>#7>> #8
\end{CD}
}

\newcommand{\xycomsquare}[8]                   
{\xymatrix
{#1 \ar[r]^{#2} \ar[d]^{#4} &
#3 \ar[d]^{#5}  \\
#6\ar[r]^{#7} &
#8
}
}

\newcommand{\IR}{{\mathbb R}}

\newcommand{\higherlim}[3]{{\setbox1=\hbox{\rm lim}
        \setbox2=\hbox to \wd1{\leftarrowfill} \ht2=0pt \dp2=-1pt
        \mathop{\vtop{\baselineskip=5pt\box1\box2}}
        _{#1}}^{#2}#3}

\newcommand{\mcg}[1]{\mathrm{MCG}^{*}(#1)}         
\newcommand{\cgr}[1]{\mathcal{C}(#1)}              

\newcommand{\version}[1]                       
{\begin{center} last edited on #1\\
last compiled on \today\\
name of texfile: \jobname
\end{center}
}

\newcounter{commentcounter}

\begin{document}

\maketitle
\begin{center}
\end{center}

\begin{abstract}    
 We explore the use of Costa and Farber's model for random simplicial complexes to give probabilistic evidence for exhaustion via rigid expansions on  random simplicial complexes  which  are analogous of curve complexes. This has potential applications to action rigidity following Ivanov's meta-conjecture.
\end{abstract}

\section{Introduction}

In the  last  few years, probabilistic  methods  became  an  important  tool  for  the  study  of  problems  of  geometric  nature. These problems  have  their  origin  in  the  study of  several  models  of  random  groups  \cite{ollivier}, \cite{duchinlelievre} \cite{gromovrandomgrp}  and  they try  to  include  a  rich  variety  of  arguments   from  the theory  of  random  graphs \cite{bollobas} \cite{kahleclique}, \cite{kahlepittel}. 

Stochastic  Algebraic  Topology, and  specifically  the  study  of  several models  of  random  simplicial  complexes has  emerged  as an  important  field  of  research  in the  intersection  of  Geometric Group  Theory, Measurable Group Theory,  Algebraic Topology and  Probability. Relevant  for  this  work  is  the introduction by A. Costa  and M.  Farber of  the vast  generalization  of a multiparametric random  simplicial  complex  \cite{costafarber1}, \cite{costafarber2}, \cite{costafarber3}. 
 
The  curve  complex of an orientable surface $S$, first introduced by W.J. Harvey  \cite{harvey} and denoted by $\cgr{S}$, is  the abstract simplicial complex  whose vertices are the  isotopy  classes  of essential simple  closed  curves in $S$,  and  whose simplices are defined by disjointness. See Section \ref{section:expansion} for more details. It is clearly a flag complex, and its $1$-skeleton is known as the \textit{curve graph}.

It is well-known that the  curve  complex  of  $S$ has  the  following  properties: 
\begin{enumerate}
\item It  is  connected. 
\item The vertices of the curve graph  have  infinite  degree.
\item If $S$ is a closed surface of genus $g \geq 2$, the curve graph has  clique number bounded  above  by  $3g-3$.
\item If $S$ is a surface of finite topological type (i.e. its fundamental group is finitely generated), then both the curve graph and the curve complex have infinite diameter and are  Gromov-hyperbolic  as  metric  spaces.
\end{enumerate}

The \textit{extended mapping class group} of $S$, denoted by $\mcg{S}$, is the group of isotopy classes of self-homeomorphisms of $S$. There is a natural action of $\mcg{S}$ on $\cgr{S}$ by simplicial automorphisms, with representation $\Psi: \mcg{S} \to \mathrm{Aut}(\cgr{S})$ given by $[h] \mapsto ([\alpha] \mapsto [h(\alpha)])$. This representation is an isomorphism for most of the surfaces. See \cite{farbmargalit}, \cite{hernandezmoralesvaldez}, \cite{celebratedIvanov},\cite{celebratedKorkmaz}, \cite{celebratedLuo}, \cite{hernandezvaldez}, \cite{hernandezmoralesvaldez2}, \cite{bavarddowdallrafi}, \cite{brendlemargalit}.

In the last couple of decades Ivanov's result for the curve complex has been generalised to more general simplicial/graph morphisms. Of particular interest for this work is the work of J. Aramayona and C.J. Leininger on rigid subgraphs of the curve graph; see \cite{aramayonaleiningerfinite} and \cite{aramayonaleiningerrigid}. Here we   examine  a   question of  geometric  nature, motivated  by  the  notion  of  \emph{rigid  expansion}   in the  curve  complex, in the  sense  of \cite{aramayonaleiningerfinite}, \cite{aramayonaleiningerrigid} and  particularly  the  graph theoretic reformulation of  \cite{hernandezrigidity}  with  the  methods  of Stochastic  Algebraic Topology.   

One of the motivations for rigid expansions is that the  proces  of rigid  expansion allows to create rigid supergraphs out of rigid subgraphs. Moreover, with this method one can exhaust the curve graph, which leads to several results concerning simplicial rigidity of the curve graph. This method  was originally introduced   by Aramayona-Leininger \cite{aramayonaleiningerrigid}, and   in  the specific  situation  that  we  study  here  by  the third author in \cite{hernandezexhaustion} and \cite{hernandezrigidity}. We review this concept in Section \ref{section:expansion}. 

In \cite{hernandezexhaustion}, the third author proved that the curve graph of a closed surface of genus at least three can be exhausted via rigid expansions by a finite subgraph consisting of a closed chain of length $2g+2$, and a system of external vertices. In Section \ref{section:expansion} we prove we only need the closed chain and two external vertices. 

After translating the previous situation to purely combinatorial terms, in this work we study how common finding this subgraphs is by defining a random variable on the multiparametric model for random simplicial complexes of Costa and Farber.

We  formulate  the problem in  the  expectation of  a  random variable which  counts  \emph{closed chains  of  length 2g+2 with  a system of  two ``alternating''  external  vertices inside a 2g+4 simplex.}

The  outline   of  the  main  result  in  this  note is  the  following. Take  a  random  simplicial  complex  $Y\subset \Omega ^{r,\mathfrak{P}}_{n}$ with at most  $n=2^{g}>0$ vertices and  vector  of  probabilities $\mathfrak{P}= (p_{0},p_1, \ldots, p_{n} )$  in the multiparametric Model, with  $p_{i}= n^{-\alpha_{i}}$, where $\alpha_i$  is  a function  of  $n$,  such  that  $\alpha_{i}$  has  a  limit  as  $n$ tends  to  infinity. See Section \ref{section:multiparametricmodel} for more details.

Assume  that  the  simplicial  complex  is  nonempty,  connected  and  Gromov hyperbolic. Recall  that   as  a  consequence  of  Theorem  5 in \cite{costafarber2},  this  is  granted  whenever the  following set  of  inequalities  hold: 

\begin{condition}\label{condition:curve}[Conditions to  model  the  curve  complex]
\begin{itemize}
\item  For  the  simplicial  complex  to  be  hyperbolic: $\alpha_{0}+ 3 \alpha_{1}+ 2\alpha_{2}> 1$ with $\alpha_{2}>0$, and  with $0<\alpha_{0}+ \alpha_{1} <1$ to  be  connected.   (Theorem 5  in page  449  of \cite{costafarber2}).  

\end{itemize}
\end{condition}

The  following condition is  the  fundamental  threshold  for  the  main  statements of  this  article  to   hold  or not  asymptotically  almost  surely :

\begin{condition}\label{condition:critical dimension}
Assume that  the parameters  for  the  random  simplicial complex  satisfy  the  hypothesis  of  critical dimension in the  sense  of  Costa  and  Farber \cite{costafarber3}. 

\end{condition} 

The  condition  of  critical  dimension  was  originally defined  in  the  context  of phase  transitions  in homology of  random  simplicial  complexes. 

The  reason  behind  the  use  of  this  condition  is  that  the  argument  that  we  present   consists  of  analyzing  the  expectation  of  a   random variable  related  to  a geometric condition  inside  a   complete  $4g+2$ graph which  is  a subcomplex  of  the  random  simplicial  complex,   and  the  conditions \ref{condition:curve} pose restrictions on  the  first  three  parameters,   the  critical dimension  gives  control  of the  remaining  parameters  in  estimates for  the  expectation and  variance  of  a random variable $CH$.     

Precisely,  we  denote  by $CH$ the discrete    random variable   counting  closed  chains  of length $2g+2$  with an alternating  pair  of  external  vertices inside a $2g+4$ complete  subgraph. See section \ref{section:estimates} for  precise  definitions.

\begin{theorem}\label{theo:countingfavorable}
Assume $Y$ is  a $(3g-3)$-dimensional random simplicial  complex   and satisfies Conditions \ref{condition:curve} and \ref{condition:critical dimension}, as  well as  the  technical  condition\ref{condition:technical}. 
  
Then  the discrete random  variable $CH$ tends  to  infinity  as  $g$  tends  to  infinity. 
\end{theorem}

This result  can  be  interpreted as  a  statement giving  probabilistic  evidence for  the Costa-Farber multiparametric model to be a geometric/simplicial limit for the curve complexes of closed surfaces when the genus tends to infinity.

 In this sense, for studying several assymptotic properties of a geometric and large-scale nature, the Costa-Farber model with  specific  parameters is better suited than the curve complex of an infinite-genus surface, since in the latter case the curve complex has finite diameter.  

Other  developments  in  this  direction, which  complement  the  view on this  problems  include  the  result  by Bering  and  Gaster  that  the  curve  complex  of  an infinite  genus  surface  contains  the  random  graph (which  is  the  limit  of  an  Erd\"os-R\'enyi Graph) in \cite{beringgaster}, as  well  as  the  definition   by  Farber, Mead  and  Strauss  of  the  Rad\'o  simplicial  complex, including  the recent observation  by  Farber  \cite{farbermeadstrauss},   that  the  Rad\'o  complex is highly  symmetric, in the  sense  that local injective automorphisms extend  to  global simplicial  automorphisms.

\subsection*{Acknowledgments}
The first  author  thanks  the  support  by  PAPIIT  Grant IA 100119 and  IN 100221, as  well  as a  Sabbatical  Fellowship by DGAPA-UNAM for  a  stay  at the  University  of  the  Saarland and  the  SFB TRR 195 Symbolic Tools in Mathematics and their Application. 

The computational  experimentations  giving  evidence   for  the  main  result  were  obtained  by  the  second  named author  in  his  Ms Sc. Thesis.

The authors  thank  enlightening  conversations  with  Octavio Arizmendi, who helped  to  settle  the  first  estimates  for  the  expectation  of  a random variable  in  graphs leading  to  Theorem \ref{theo:countingfavorable}.

The third author was supported during the creation of this article by the research project grants UNAM-PAPIIT IA104620 and UNAM-PAPIIT IN102018.

Finally, the   original  idea  of  looking for  probabilistic evidence  for  rigidity in  general  came as  a  result  of  the  interaction   initiated  by the  Mexican  Network  of  Topological Data Analysis  and Stochastic Topology, \url{https://atd.cimat.mx/}, which gave  the  authors  the  opportunity  to   have  enlightening conversations  with  Michael Farber and  M. Kahle whom the  authors  thank  for many interesting  viewpoints on the  development  of Stochastic Topology.


\section{The  curve  complex and  rigid  expansions. }\label{section:expansion}

Let $S$ be an orientable surface. A \textit{curve} is a topological embedding of the unit circle into $S$. We often abuse notation and call ``curve'' the embedding, its image on $S$ or its isotopy class. The context makes clear which use we mean.

A curve is \textit{essential} if it is neither null-homotopic nor homotopic to the boundary curve of a neighbourhood of a puncture.

The \textit{(geometric) intersection number} of two (isotopy classes of) curves $\alpha$ and $\beta$ is defined as follows: $$i(\alpha,\beta) := \min \{|a \cap b| : a \in \alpha, b \in \beta\}.$$

The \textit{curve complex of} $S$, denoted by $\cgr{S}$ is the abstract simplicial complex whose vertices are the isotopy classes of essential curves on $S$, and the set $\{\alpha_{0}, \ldots, \alpha_{k}\}$ is a $k$-simplex if for all $i,j \leq k$ we have that $i(\alpha_{i},\alpha_{j}) = 0$. Note that $\cgr{S}$ is a flag complex. The 1-skeleton of $\cgr{S}$ is called the \textit{curve graph of} $S$.

A result due to Ivanov \cite{celebratedIvanov}, Korkmaz \cite{celebratedKorkmaz} and Luo \cite{celebratedLuo}, asserts that, excluding finitely-many well-understood cases, the action of the extended mapping class group is effective, and the automorphisms of the curve complexes are all geometric, i.e. they are induced by a homeomorphism. This means that the group $Aut(C(S))$ of simplicial automorphisms of $C(S)$ is isomorphic to the extended mapping class group $\mcg{S}$. 

In \cite{aramayonaleiningerfinite} and \cite{aramayonaleiningerrigid}, the concept of a subgraph of the curve graph being rigid was introduced. Later, in \cite{hernandezrigidity}, this definition was generalised to the context of simplicial graphs. The following definition is an obvious generalisation to the context of abstract simplicial complexes.

\begin{definition}
Let $\Gamma$ be an abstract simplicial complex and $Y<\Gamma$ be a vertex-induced subcomplex. We say $Y$ is \emph{rigid} if every locally injective map $Y \to \Gamma$ can be extended to an automorphism of $\Gamma$.
\end{definition}

Recall that a map $f:Y \to \Gamma$ is \emph{locally injective} if $f|_{star(v)}$ is injective for all $v \in V(Y)$, where $\Gamma$ is a simplicial complex, $Y$ is a vertex-induced subcomplex, and $star(v)$ is the subcomplex induced by all the simplices containing $v$ as a vertex.

In general, if $Y < X < \Gamma$ are vertex-induced subcomplexes and $Y$ is rigid, there is no reason why $X$ should also be rigid. In \cite{aramayonaleiningerfinite}, a method for enlarging subgraphs was  developed to solve this issue, and in \cite{hernandezrigidity} it was also generalised to simplicial graphs. Here we present the obvious analogue to abstract simplicial complexes:

\begin{definition}
Let $\Gamma$ be an abstract simplicial complex and $Y$ be a vertex-induced subcomplex of $\Gamma$. Denoting by $V(\Gamma)$ the set of vertices of $\Gamma$, and by $adj(v)$ the set of vertices of $\Gamma$ that span a 1-simplex with $v$, we have the following definitions:
 \begin{enumerate}
  \item Let $A \subset V(\Gamma)$ and $v \in V(\Gamma)$. We say $v$ is \textit{uniquely determined by} $A$, denoted by $a = \langle A \rangle$ if they satisfy that $$\{v\} = \bigcap_{w \in A} adj(w).$$
  \item The \textit{zeroth rigid expansion of} $Y$, denoted by $Y^{0}$ is defined as $Y$.
  \item If $\alpha$ is a countable ordinal that is the successor of $\beta$, then the $\alpha$\textit{-th rigid expansion of} $Y$, denoted by $Y^{\alpha}$ is the vertex-induced subcomplex whose vertices are $$V(Y^{\alpha}) := V(Y^{\beta}) \cup \{v \in V(\Gamma): \exists A \subset V(Y^{\beta}) \text{ with } v = \langle A \rangle \}.$$
  \item If $\alpha$ is a countable limit ordinal, then the $\alpha$\textit{-th rigid expansion of} $Y$, denoted by $Y^{\alpha}$ is the vertex-induced subcomplex whose vertices are $$V(Y^{\alpha}) := \bigcup_{\beta < \alpha} V(Y^{\beta}).$$
 \end{enumerate}
\end{definition}

Using the same arguments as in Section 1 of \cite{hernandezrigidity} and the same proof as Proposition 3.5 in \cite{aramayonaleiningerrigid} we can easily verify that rigid expansions satisfy the following result.

\begin{proposition}\label{proposition:rigidity and expansions}
 Let $\Gamma$ be an abstract simplicial complex, $Y < \Gamma$ be a rigid subcomplex, and $\alpha$ be a countable ordinal. Then $Y^{\alpha}$ is a rigid subcomplex of $\Gamma$.
\end{proposition}

Now, let $Y < \Gamma$ be a vertex-induced subcomplex, we say $Y$ is a \textit{seed subcomplex} if there exists a countable ordinal $\alpha$ such that $Y^{\alpha} = \Gamma$. In \cite{hernandezexhaustion} the third author proved the following result concerning the curve graph of a closed surface.

\begin{theorem}[3.1 in \cite{hernandezrigidity}]\label{theorem:exhaustioncurvegraph}
 Let $S$ be an orientable closed surface of genus at least 3. Then there exists a finite seed subcomplex $Y < \cgr{S}$.
\end{theorem}

To construct the subcomplex $Y$ exhibited in \cite{hernandezexhaustion} we need the following definitions.

Let $k \in \mathbb{Z}^{+}$ and $C = \{\gamma_{0}, \ldots, \gamma_{k}\}$ be an ordered set of $k+1$ curves in $S$. We call $C$ a \textit{chain} of length $k+1$ if $i(\gamma_{i},\gamma_{i+1}) = 1$ for $0 \leq i \leq k-1$, and $\gamma_{i}$ is disjoint from $\gamma_{j}$ for $|i - j| > 1$. On the other hand, $C$ is called a \textit{closed chain} of length $k+1$ if $i(\gamma_{i},\gamma_{i+1}) = 1$ for $0 \leq i \leq k$ modulo $k+1$, and $\gamma_{i}$ is disjoint from $\gamma_{j}$ for $|i - j| > 1$ (modulo $k+1$); a closed chain has maximal length if it has length $2g+2$. A \textit{subchain} is an ordered subset of either a chain or a closed chain which is itself a chain, and its length is its cardinality.

Recalling that $k \geq 1$, note that if $C$ is a chain (or a subchain) of odd length, a closed regular neighbourhood $N(C)$ has two boundary components; we call these curves the bounding pair associated to $C$.

Let $C = \{\alpha_{0}, \ldots, \alpha_{2g+1}\}$ be the closed chain in $S$ depicted in Figure \ref{OriginalChainv2}. Observe it is a closed chain of maximal length, and given any other maximal closed chain $C^{\prime}$ there exists an element of $\mcg{S}$ that maps $C^{\prime}$ to $C$ (see \cite{farbmargalit}). Then, we define the set $B(C)$ as the union of the bounding pairs associated to the subchains of odd length of $C$. The seed subcomplex $Y$ of Theorem \ref{theorem:exhaustioncurvegraph} is $C \cup B$.

\begin{figure}[h]
\begin{center}
 \resizebox{10cm}{!}{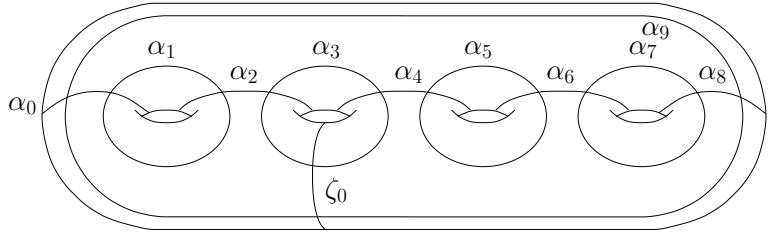} \caption{The set $C = \{\alpha_{0}, \ldots, \alpha_{2g+1}\}$ and an example of a curve $\zeta_{0}$.} \label{OriginalChainv2} 
\end{center}
\end{figure}

An immediate result coming from the definition of rigid expansions is that if $Y < X < \Gamma$ are vertex-induced subcomplexes, then for every countable ordinal $\alpha$, we have that $Y^{\alpha} < X^{\alpha}$. For the curve complex of an orientable closed surface $S$ of genus at least 3, this implies that if $Z$ is a vertex-induced subcomplex of $\cgr{S}$ such that there exists a countable ordinal $\alpha$ that satisfies $C \cup B < Z^{\alpha}$, then $Z$ is a seed subcomplex of $\cgr{S}$. As such, we have the following result.

\begin{lemma}\label{theo:jesusglasgow}
 Let $S$ be an orientable closed surface $S$ of genus at least 3, $C = \{\alpha_{0}, \ldots, \alpha_{2g+1}\}$ be a closed chain of maximal length, $\zeta_{0}$ be a curve on $S$ disjoint and different from all $\alpha_{i}$ with $i$ even, $\zeta_{1}$ be a curve on $S$ disjoint and different from all $\alpha_{i}$ with $i$ odd. Then the subcomplex of $\cgr{S}$ induced by the vertices $C \cup \{\zeta_{0}, \zeta_{1}\}$ is a seed subcomplex.
\end{lemma}

\begin{proof}
 Given $C$ as in the hypotheses, we define $C_{0} = \{\alpha_{i} \in C:\text{ } $i$ \text{ is even}\}$ and $C_{1} = C - C_{0}$. Let $\Sigma_{0}^{\pm}$ be the two connected components of $S - C_{0}$, and $\Sigma_{1}^{\pm}$ be the two connected components of $S-C_{1}$. Since for $i = 1,2$, $\zeta_{i}$ is disjoint and different from every element in $C_{i}$, we have that $\zeta_{i}$ must be contained in either $\Sigma_{i}^{+}$ or $\Sigma_{i}^{-}$; without loss of generality we can assume it is contained in $\Sigma_{i}^{+}$.
 
 Now, if $C^{\prime} = \{\alpha_{j}, \ldots, \alpha_{j+2k}\}$ is a subchain of $C$ of odd length with $j$ of the same parity as $i$, and we denote by $\beta^{+}$ and $\beta^{-}$ the curves that constitute the bounding pair associated to $C^{\prime}$, then by the same reasoning as above we have that (up to relabeling) $\beta^{\pm}$ is contained in $\Sigma_{i}^{\pm}$.
 
 Thus, there exists a subchain $C^{\prime}$ of $C$ such that $\beta^{+}$ intersects $\zeta_{i}$. Then we have the following: $$\beta^{-} = \langle \{\zeta_{i}\} \cup C - \{\alpha_{j-1},\alpha_{j+2k+1}\} \rangle \in Z^{1},$$ $$\beta^{+} = \langle \{\beta^{-}\} \cup C - \{\alpha_{j-1}, \alpha_{j+2k+1}\} \rangle \in Z^{2}.$$
 
 Repeat the process for every $\beta^{+}$ that intersects $\zeta_{i}$, and then substitute $\zeta_{i}$ by all the $\beta^{\pm}$ obtained this way. Since for every $\beta^{+}$ and $(\beta^{+})^{\prime}$, contained in $\Sigma_{i}^{+}$ either they intersect or there exists a $(\beta^{+})^{\prime\prime}$ that intersects both, we have that all curves in $B$ that are disjoint and different from $C_{i}$ are contained in $Z^{5}$. Thus $B$ is contained in $Z^{5}$ which implies that $C \cup B$ is contained in $Z^{5}$, and by the argument above we have that $Z$ is a seed subcomplex.
\end{proof}

\subsection{Translation to combinatorial terms}\label{subsection:translation}

Now that we have proved that $C \cup \{\zeta_{0},\zeta_{1}\}$ is a seed subcomplex of $\cgr{S}$, we dedicate this subsection to describe it in combinatorial terms. To do this, we need to describe geometric intersection $0$ and $1$ in combinatorial terms.

For the rest of this subsection, let $\Gamma$ be an abstract simplicial complex and $S$ be a closed surface of genus $g \geq 3$.

We say two vertices $v, w \in V(\Gamma)$ have \textit{intersection} $0$ if $\{x,y\}$ is a simplex in $\Gamma$.

To encode intersection $1$ combinatorially, we need to encode several other concepts first.

Let $v, w \in V(\Gamma)$ be two different vertices, and $\sigma$ be a top-dimensional simplex in $\Gamma$ with $v \in \sigma$. We say $v$ \emph{can be exchanged with $w$ with respect to $\sigma$} if $(\sigma \backslash \{v\}) \cup \{w\}$ is a top-dimensional simplex in $\Gamma$. In the context of $\cgr{S}$, this means we have a pants decomposition and substitute one of its curves with another with the resulting set being still a pants decomposition. In particular, this means that the two curves are contained in a subsurface of complexity one (i.e. in a one-holed torus or a four-holed sphere) whose boundary curves are elements of the pants decomposition.

Let $P$ be a pants decomposition of $S$, and $\gamma_{1}, \gamma_{2} \in P$ be two different curves. We say \emph{$\gamma_{1}$ and $\gamma_{2}$ are adjacent with respect to $P$} if there exists a closed subsurface $\Sigma$ whose interior is isomorphic to a thrice-punctured sphere, such that $\Sigma$ has $\gamma_{1}$ and $\gamma_{2}$ as two of its boundary curves. The \emph{adjacency graph of $P$}, denoted $\mathcal{A}(P)$, is the simplicial graph whose vertex set is $P$, and two vertices span an edge if the curves they represent are adjacent with respect to $P$. This graph was first introduced by Behrstock and Margalit in \cite{behrstockmargalit} and has been extensively used in works related to the combinatorial rigidity of the curve graph. Note that in the proof of Lemma 7 in \cite{shackleton}, Shackleton implicitly gives a characterisation of being adjacent: $\gamma_{1}$ and $\gamma_{2}$ are adjacent with respect to $P$ if and only if there exists curves $\delta_{1}$ and $\delta_{2}$ such that $i(\delta_{1},\delta_{2}) \neq 0$ and for $i = 1,2$ we have that $(P \backslash \{\gamma_{i}\}) \cup \{\delta_{i}\}$ is a pants decomposition. Based on this characterisation we give the following definitions.

Let $v_{1}, v_{2} \in V(\Gamma)$ be two different vertices, and $\sigma$ be a top-dimensional simplex in $\Gamma$ with $v_{1}, v_{2} \in \sigma$. We say \emph{$v_{1}$ and $v_{2}$ are adjacent inside $\sigma$} if there exist $w_{1}, w_{2} \in V(\Gamma)$ such that $\{w_{1},w_{2}\}$ is not a simplex in $\Gamma$, and for $i=1,2$ we have that $v_{i}$ can be exchanged with $w_{i}$ with respect to $\sigma$.

Following the same analogy as above, given $\sigma$ a top-dimensional simplex in $\Gamma$ the \emph{adjacency graph of $\sigma$}, denoted $\mathcal{A}(\sigma)$ is the simplicial graph whose vertex set is $\sigma$, and two vertices span an edge if they are adjacent inside $\sigma$.

Recall that a curve $\alpha$ in $S$ is called \emph{separating} if $S \backslash \alpha$ is not connected, and is called \emph{non-separating} otherwise.

Now, let $P$ be a pants decomposition, and $\gamma_{1}, \gamma_{2} \in P$. It is clear that, since $S$ is closed, a vertex in $\mathcal{A}(P)$ is a cut vertex if and only if it correspond to a separating curve of $S$. Thus, since $S$ is a closed surface with genus at least three we have the following: if $\gamma_{1}$ is a leaf in $\mathcal{A}(P)$ and $\gamma_{2}$ is the unique vertex in $\mathcal{A}(P)$ adjacent to $\gamma_{1}$ with respect to $P$, then $\gamma_{2}$ is a separating curve that bounds a one-holed torus and $\gamma_{1}$ is a non-separating curve contained in said torus.

Analogously, let $\sigma$ be a top-dimensional simplex, and $v_{1}, v_{2} \in \sigma$. We say \emph{$v_{2}$ separates a torus containing $v_{1}$ with respect to $\sigma$} if $v_{1}$ is a leaf in $\mathcal{A}(\sigma)$ and $v_{2}$ is the unique vertex in $\mathcal{A}(\sigma)$ adjacent to $v_{1}$ inside $\sigma$.

Then, to encode geometric intersection $1$, we use Ivanov's characterisation of it (Lemma 1 in \cite{celebratedIvanov}) and use the terminology defined above.

Let $v_{1}$ and $v_{2}$ be two different vertices of $\Gamma$. We say they have \emph{intersection $1$} if there exists different vertices $v_{3}$, $v_{4}$ and $v_{5}$, and a top-dimensional simplex $\sigma$ such that:
\begin{enumerate}
 \item $\{v_{i}, v_{j}\}$ is a simplex if and only if $\{i,j\}$ is an edge in Figure \ref{Estrella}.
 \item $v_{1}, v_{4} \in \sigma$.
 \item $v_{4}$ separates a torus containing $v_{1}$ with respect to $\sigma$.
 \item $v_{1}$ can be exchanged with $v_{2}$ with respect to $\sigma$.
\end{enumerate}
\begin{figure}[h]
\begin{center}
 \resizebox{5cm}{!}{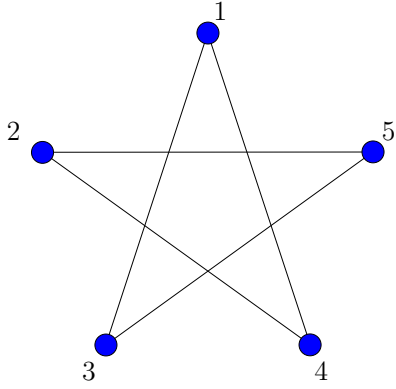} \caption{The graph used for the definition of intersection $1$.} \label{Estrella} 
\end{center}
\end{figure}

Thus, we can define a closed chain as follows: the set $\{v_{0}, \ldots, v_{k}\}$ is a \emph{closed chain of length} $k+1$ if $v_{i}$ and $v_{j}$ have intersection $0$ whenever $|i - j| > 1$, and if $v_{i}$ and $v_{j}$ have intersection $1$ whenever $|i -j| = 1$.

Finally, to encode the set used in Lemma \ref{theo:jesusglasgow} we define it as $Y = C \cup \{w_{0}, w_{1}\}$, where $C = \{v_{0}, \ldots, v_{2g+1}\}$ is a closed chain, and $w_{0}$ and $w_{1}$ are two vertices such that $\{w_{i}, v_{j}\}$ is a simplex if and only if $i$ and $j$ have the same parity.

\section{The multiparametric model  for  random simplicial complexes }\label{section:multiparametricmodel}

In this section we  recall  the definition and basic results of the multiparametric  model  for  random  simplicial  complexes due  to  Costa  and Farber \cite{costafarber1}, \cite{costafarber2}, \cite{costafarber3}. 

Let $g$ be  a  natural number,  $n=2^{g}$  and $\Delta_{n}$ be the $n$-dimensional simplex. 

We want to consider $r$-dimensional  simplicial  complexes,  where  $r =3g-3\leq n$.

Given  a simplicial  subcomplex  $Y\subset \Delta_{n}$,  we  denote  by $f_{i}(Y)$ the  number of  $i$-dimensional  simplices. 

 Also, given   $\sigma$  a simplex in $\Delta_n$, recall  that  $\sigma$  is  an  \textit{external  face of} $Y$  if $\partial \sigma \subset Y$ but $\sigma$  is  not   completely contained in  $Y$. We  denote  by  $E(Y)$ the  set  of  exterior  faces  of  $Y$ and  by  $F(Y)$  the set  of  faces  of $Y$.

We  recall  the  definition  of  the  multiparametric model  for  random simplicial complexes of  Costa  and  Farber. 

\begin{definition}
Let $r$ be  a natural number  between  $0$ and  $n$.  Let $\mathfrak{P}= (p_{0}, p_{1}, \ldots, p_{n})$ be  an $n+1$-tuple of  probabilities, that  is, real   numbers  in $[0,1]$. The space $\Omega_{n}^{r, \mathfrak{P}}$ is  the measure  space  where  the  objects  are  $r$-dimensional  simplicial  subcomplexes  of  the  $n$-dimensional  simplex $\Delta_n$,  and  where  the  probability function  is  given  by 

$$\mathbb{P}_{\mathfrak{P}}(Y)= \underset{\sigma \in F(Y)}{\prod} p_{\sigma} \underset{\sigma \in E(Y)} {\prod }q_{\sigma}.$$
Here, $p_{\sigma}= p_{i}$ if  $\sigma$  is  an $i$-dimensional  face, $q_{\sigma}= q_{i}$,  and  $q_{i}=1-p_{i}$. 
\end{definition}

The  following  result  is  proved  in \cite{costafarber1} as Lemma 2.3.  

\begin{theorem}\label{lemma:probability}
Let $A\subset B\subset \Delta_{n}$  be  simplicial  subcomplexes  such  that the  boundary  of  any  external  face of  $B$ of  dimension $\leq r$  is  contained  in A. Then, 

$$\mathbb{P}_{\mathfrak{P}} \huge  \{ A \subset Y \subset B  \huge \}= \underset{\sigma \in F(A)}{\prod } p_{\sigma} \underset{\sigma \in E(B)}{\prod} q_{\sigma}.$$
\end{theorem}

The  last  result  characterizes  the  probability measure,  in  the  sense  that  any  other probability  measure with  the  property  that if $\bf{P}$ is  some  probability  measure  for  which,  given  an arbitrary  subcomplex $A\subset Y\subset \Delta_n$, the  equality  
$$\mathbf{P}(A\subset\Delta_{n})= p_{i}^{f_{i}(A)}  $$ 
implies  that $\mathbf{P}$  is  the  measure $\mathbb{P}$.

The notion  of  critical  dimension  was  introduced  in \cite{costafarber2}. The  idea  is  that  the  probability vector $\mathfrak{P}$ of  the  random simplicial  complex  lies  between  affine  subspaces $\mathfrak{D}_{i}$, for  $i\in \{0, n\}$ where  the  homology  in degree  $i$ has  rank  significatively  bigger  than  in  any other degree. 

Consider  the linear  functions $\psi_{k}: \IR^{r+1}\to \IR$ defined  by 

$$\psi_{k} (\alpha)= \psi_{k}(\alpha_{0}, \ldots, \alpha_{r}) = \underset{i=0}{\overset{r}{\sum}} \binom{k}{i}\alpha_{i},$$

with  the  conventions $\binom{k}{i}=0$  for  $i<k$ and $\binom{0}{0}=1$. 
Since $\binom{k}{i}<\binom{k+1}{i}$,  one  has 
$$\psi_{0}(\alpha)\leq \psi_{1}(\alpha)\leq \ldots \leq  \psi_{r}(\alpha).$$ 

\begin{definition}
The  domain $\mathfrak{D}_{k} $ is  the  affine  subspace  defined  by  the  inequalities 
$$ \{ \alpha \in \IR^{n+1}\mid \psi_{k}(\alpha)<1 <\psi_{k+1}(\alpha)\}.$$
\end{definition}
 
Now, the  probabilistic  model  for  the   curve  complex  of  the  orientable  genus  $g$  closed  surface is  the  multiparametric  random  simplicial  complex,  in the  sense  of  Costa  and  Farber with  the  following  specific parameters:

 Let $r=3g-3$, consider  the  Costa and  Farber  multiparametric  model  for  random  $r$-dimensional simplicial  complexes,   and assume  that  the probability  vector $\mathfrak{P}$  is  of  the  form $\mathfrak{P}= (n^{-\alpha_{0}}, n^{-\alpha_{1}},n^{-\alpha_{2}}, \ldots, n^{-\alpha_{n}})$ for  $\alpha $  satisfying  Conditions \ref{condition:curve} and  \ref{condition:critical dimension},  which  we  state additionally here for  the  sake  of  completness:
  
 \begin{enumerate}

\item  For  the  simplicial  complex  to  be  hyperbolic,   connected and  non-empty: $\alpha_{0}+ 3 \alpha_{1}+ 2\alpha_{2}< 1$ with $\alpha_{0}+ \alpha_{1} <1$.   (Theorem 5  in page  449  of \cite{costafarber2}).  

\item  For  the simplicial  complex  to  have  critical  dimension $k=4g+2$ the parameters $\alpha_{*}$  satisfy  the  condition of  belonging  to  the  domain $\mathfrak{D}_{k}$.

\end{enumerate}  

We  state  now  the  technical  condition. 
\begin{condition}\label{condition:technical}
Assume $\alpha_{1}< \frac{1}{g^{2}}$, $\alpha_{0}<\frac{g^{2}-1}{g^{2}}$,  and $\alpha_{2}>\frac{1-2g^{2}}{g^{2}}.$
\end{condition}

Condition \ref{condition:technical} is  a  specific realization of  hyperbolicity, non  emptyness  and   connectedness  of  a random simplicial  complex  for  parameters  of  fixed  decrease  order  in  funcion  of $g$ which  appear  in  our  probabilistic  estimates. 

\section{Geometric estimates.}\label{section:estimates}

We  now give estimates for   the  geometric  condition of  a closed chain length $2g+2$ and a pair  of external  alternating vertices as in the  situation  described in Lemma \ref{theo:jesusglasgow}. 

Let $\Delta$ be the $(2g+3)$-simplex with vertex set $\{v_{0}, \ldots, v_{2g+1}\} \cup \{w_{0}, w_{1}\}$. We define $A$ as the subgraph of $\Delta$ that has the same vertex set as $\Delta$, and whose edges are defined as follows:
\begin{enumerate}
 \item $\{v_{i}, v_{j}\}$ is an edge in $A$ if and only if $|i - j| > 1$ modulo $2g+2$.
 \item $\{w_{i}, v_{j}\}$ is an edge in $A$ if and only if $i$ and $j$ have the same parity.
\end{enumerate}
See Figure \ref{figure:Grafo-A} for an example with $g = 4$.

\begin{center}
\begin{figure}[ht]
 \resizebox{10cm}{!}{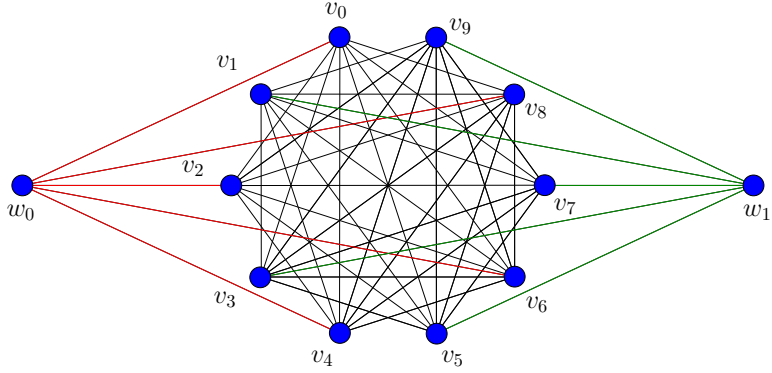} \caption{The subgraph $A$ for the case $g = 4$.} \label{figure:Grafo-A} 
\end{figure}
\end{center}

As it can be deduced from Subsection \ref{subsection:translation}, any seed graph given by Lemma \ref{theo:jesusglasgow} has $A$ as a subgraph of its subjacent graph, but not every subgraph isomorphic to $A$ comes from a seed graph given by Lemma \ref{theo:jesusglasgow}.

Now, if $Y$ is the subjacent subgraph of a seed graph given by Lemma \ref{theo:jesusglasgow}, it may happen that $Y$ had more edges than $A$. Since the vertices $v_{0}, \ldots, v_{2g+1}$ represent a closed chain in the surface, if $Y$ has more edges than $A$, they have to be edges of the form $\{w_{0}, v_{2k+1}\}$  for some $k$, $\{w_{1}, v_{2k}\}$  for some $k$, or $\{w_{0}, w_{1}\}$. Since we can create examples of $Y$ with edges of these forms for each value of $k$ from $k = 0$ to $k = g$, we have to consider the possibility that our random complex may have any of these edges too.

For this reason, we define $B$ to be the flag complex induced by the subgraph of $\Delta$ with vertex set $\{v_{0}, \ldots, v_{2g+1}\} \cup \{w_{0}, w_{1}\}$, and whose edges are the following:
\begin{enumerate}
 \item $\{v_{i}, v_{j}\}$ is an edge in $B$ if and only if $|i - j| > 1$ modulo $2g+2$.
 \item $\{w_{i}, v_{j}\}$ is an edge in $B$ for all $j$.
 \item $\{w_{0}, w_{1}\}$.
\end{enumerate}

 Thus, if $Y$ is the subjacent subgraph of a seed graph given by Lemma \ref{theo:jesusglasgow}, we have that $A \subset Y \subset B$. Now, we compute  the  probability  of  this  event  happening.

\begin{lemma}
Let $Y\in \Omega_{n}^{r, \mathfrak{P}}$ be  a  random simplicial  complex  in  the  multiparametric model. Then, the  probability  of  the  event  described  by  $A$ being contained in $Y$ and $Y$ being contained in $B$ is  given  by 
$$p_{0}^{2g+4} \cdot p_{1}^{2g^{2}+3g+1} (1-p_{0})^{0}( 1-p_{1})^{2g+2}.  $$
\end{lemma}

\begin{proof}
Recall  from Theorem \ref{lemma:probability}, that the probability  of  a  subcomplex $Y$ to   appear  in the  sequence $A\subset Y \subset B$  is  given  by 
$$  \mathbb{P}_{\mathfrak{P}} \huge  \{ A \subset Y \subset B  \huge \}= \underset{\sigma \in F(A)}{\prod } p_{\sigma} \underset{\sigma \in E(B)}{\prod} q_{\sigma}, $$
Where $F(A)$  denote  the  faces  of  $A$  and  $E(B)$ denote  the  external  faces  of  $B$.

This  is   the  expression 
$$ \underset{i=0}{\overset{r}{\prod }} p_{i}^{f_{i}(A)} \cdot  \underset{i=0}{\overset{r}{\prod}} q_{i}^{e_{i}(B)}, $$
where $r $ denotes the  dimension  of $B$. In our  case,  $r=g+2$, $f_{0} = 2g+4$, $e_{0} = 0$, and  we use a counting argument to   determine $f_{1}$  for  the   diagram A, and $e_{1}$  for  the   flag complex $B$. Also we argue why $e_{i} = 0$ for $i \geq 2$.

Notice  that in $A$ there are $2g^{2} +g -1$  edges spanned between the vertices $v_{0}, \ldots, v_{2g+1}$ (these are the diagonals of a $(2g+2)$-gon). Also, there are $g+1$  edges  joining  $w_{0}$ to the vertices $v_{i}$ with even $i$, and $g+1$ joining $w_{1}$ to the vertices $v_{i}$ with odd $i$.  Thus $f_{1}(A)= (2g^{2} +g -1)+(g+1)+(g+1)= 2g^{2}+3g+1$. 

For  computing $e_{1}(B)$ notice  that  the exterior edges  of $B$ are exactly those which are in $\Delta$ but not $B$. By definition of $B$, those edges are exactly the sides of the $(2g+2)$-gon with vertices $v_{0}, \ldots, v_{2g+1}$. This implies $e_{1} = 2g+2$.

Since $B$ is a flag complex, if there is a complete subgraph with at least $3$ vertices, then the corresponding simplex is contained in $B$. As such, $B$ does not have any exterior faces of dimension greater or equal to $2$. Hence, $e_{i} = 0$ for all $i \geq 2$.

Substituting  in the  formula $ \underset{i=0}{\overset{r}{\prod }} p_{i}^{f_{i}} \cdot  \underset{i=0}{\overset{r}{\prod}} q_{i}^{e_{i}(B)}$ gives  the  desired  expression. 
\end{proof}

\section{Proof  of  theorem  1 }
We  define  now  the  random  variable  which is   the   main  object  of  study. 

\begin{definition}\label{def:closedchain}
Let $CH$ be  the   discrete random variable defined on the  probability space $\Omega_{n}^{r, \mathfrak{P}}$ which  is  one   in  the  case   of  an appearance of a diagram $A$   inside  a  $(2g+4)$- complete  graph  as  subsimplicial  complex. 
\end{definition}

We  now  give  the  proof  of  Theorem \ref{theo:countingfavorable}.

\begin{proof}
The  proof  consists  of  two  arguments
\begin{itemize}
\item Prove  using  the  estimates  of Section  \ref{section:estimates}  that   the  expectation  of $CH$  tends  to  infinity as $g$  tends  to  infinity.

\item Prove  using  a  second  moment  argument  that  the  random variable is assymptotically  almost  surely  positive. 

\end{itemize}

Recall  from \cite{costafarber2}, Theorem 1  in page 444,  that  the  expectation  of  the  random variable  which  counts  embeddings  of  a  complete graph in  $4g+2$ vertices  is: 

$$\binom{n}{4g+2} (4g+2)! \underset{i=0}{\overset{3g-3}{\prod}} p_{i}^{\binom{4g+1}{i}}.$$

We  claim  that  for  $g>1$, $n= 2^{g}$  and the  parameters $p_{i}= \frac{1}{ n^{\alpha_i}}$ satisfying  the  critical  dimension  condition, \ref{condition:critical dimension},  the  logarithm in  base $n= 2 ^{g}$, this  expression  is  greater  than $6$.

The  given  expression  is  greater  than 
$$ \frac{2^{g}}{4g+2}^{4g+2}(4+2)! 2^{\psi_{k}(\alpha)} > \frac{1}{2}(4g+1!)(2^{4g^{2}+2g}). $$
which  has  logarithm base  $2^{g}$  greater than  7  if  $g \geq 2$.

Thus,  we  have  the  following estimate  for  the  expectation   of  the  random variable  $CH$

\begin{multline}\label{equation:infinite}
\mathbb{E}[CH(Y)]\approx  \binom{n}{4g+2} (4g+2)!\\ \underset{i=0}{\overset{3g-3}{\prod}} p_{i}^{\binom{4g+1}{i}} p_{0}^{2g+4} \cdot p_{1}^{2g^{2}+3g+1} (1-p_{0})^{0}( 1-p_{1})^{2g+2}..  
\end{multline}

which  tends  to  $\infty$. 

Let  us  verify  this  fact. 

Consider  $\log_{n} \mathbb{E}[CH(Y)]$. 
First  notice  that  due  to  the fact  that the critical  dimension equals $4g+2$, 
Approximation  \ref{equation:infinite} becomes then

\begin{multline}\label{equation:infinite2}
\log_{n}(\mathbb{E}[CH(Y)])\geq  \log_{n}\left(\frac{n}{4g+2}^{4g+2}\right) \\ -1 - \alpha_{0}(4g+2) - \alpha_{1}(2g^{2}+ 3g+1) + \log_{n}\left(1-\frac{1}{n^{\alpha_{1}}}^{(2g+2)}\right) 
\end{multline}

Under  the  hypothesis  of  critical  dimension, we  have  verified  that  the  first  summand  is  at  least 7. The technical condition  \ref{condition:technical} permits    to  bound  this  expression  from  below  by 

$$ 7-5 - \frac{2}{g}-2+ \frac{3}{g}+\frac{1}{g^{2}}+   \log_{n}\left(1-\frac{1}{n^{\alpha_{1}}}^{(2g+2)}\right).$$

 We  handle  now  the  last  summand.
 Recall  that  from  the  Taylor  expansion  around  zero, 
 $\log(1-z)\approx \sum_{i=0}^{\infty} -\frac{z^{i}}{i} $,  and  thus the  last   summand has  the  assymptotic  expansion 
 $$(2g+2)\left[ \sum_{i=0}^{\infty}2^{-g\alpha_{1} i}  \right].$$

From  the  technical  condition \ref{condition:technical},  we  have that $\alpha_{1}<\frac{1}{g^{2}}$,  and  thus  the  summand  $(4g+5)[ (\sum_{i=0}^{\infty}2^{-g\alpha_{1} i})  ]$ Is  bounded  below   by  the  espression 

$$ (2g+2) \sum_{i=1}^{\infty} 2^{-g}, $$
which  diverges  to  $\infty$. 

We  now  finish the proof  recalling  that  the  variance   of  the random variable $CH$  is  the  same  as  the variance   of  the  random variable which  counts $(4g+1)$- dimensional  faces  in  a  random  simplicial  complex, denoted  by $f_{4g+1}$  in  \cite{costafarber3}. It  is  proved  there  that  there  exists  a  constant $C$  such that 
$$ \frac{ \mathbb{V}(f_{4g+1})}{{\mathbb{E}(f_{4g+1})}^2 } <C n^{-\delta_{4g+1}(\alpha_{*})/2},$$ 
where  $\delta_{4g+1}(\alpha_{*})$  is  the  minimum between $\tau_{0}(\alpha_{*})$ and $\tau_{4g+1}(\alpha_{*})$. It  follows by  the   work  of  Costa  and  Farber  there  that it  tends  to  $0$  as  $g$  tends  to  infinity under condition \ref{condition:critical dimension}.  
\end{proof}

\bibliographystyle{abbrv}
\bibliography{manual}
\end{document}